\documentclass[12pt]{article}
\usepackage{mathrsfs}

\usepackage{t1enc}
\usepackage[latin1]{inputenc}
\usepackage[english]{babel}
\usepackage{amsmath,amsthm}
\usepackage{amsfonts}
\usepackage{latexsym}
\usepackage{graphicx}
\usepackage[natural]{xcolor}
\theoremstyle{plain}
\newtheorem{theorem}{Theorem}

\newtheorem{lemma}{Lemma}[section]

\newtheorem{claim}{Claim}
\newtheorem{case}{Case}
\newtheorem{subcase}{Subcase}[case]
\newtheorem{fact}{Fact}

 \renewenvironment{proof}{
 \noindent{\bf Proof.}\rm} {\mbox{}\hfill\rule{0.5em}{0.809em}\par}

\usepackage{caption,subcaption}

\begin{document}
\date{}
\title{Gallai-Ramsey numbers for graphs with five vertices and eight edges
}
\author{\small Xueli Su, Yan Liu\\
\small School of Mathematical Sciences, South China Normal University,\\
 \small Guangzhou, 510631, P.R. China  \thanks{This work is supported
by the Scientific research fund of the Science and Technology Program of Guangzhou, China(No.202002030183), by the Natural Science Foundation of Qinghai, China (No.2020-ZJ-924).
 Correspondence should be addressed to Yan
Liu(e-mail:liuyan@scnu.edu.cn)}}

\maketitle

\setcounter{theorem}{0}

\begin{abstract}
A Gallai $k$-coloring is a $k$-edge coloring of a complete graph in which
there are no rainbow triangles. For given graphs $G_1, G_2, G_3$
and nonnegative integers $r, s, t$ with that $k=r+s+t$, the
$k$-colored Gallai-Ramsey number $gr_{k}(K_{3}: r\cdot G_1,~
s\cdot G_2, ~t\cdot G_3)$ is the minimum integer $n$ such that every Gallai
$k$-colored $K_{n}$ contains a monochromatic copy of $G_1$ colored
by one of the first $r$ colors or a monochromatic copy of $G_2$
colored by one of the middle $s$ colors or a monochromatic copy of $G_3$
colored by one of the last $t$ colors. In this paper, we
determine the value of Gallai-Ramsey number in the case that
$G_1=B_{3}^{+}$, $G_2=S_{3}^+$ and $G_3=K_3$. Then the Gallai-Ramsey number
$gr_{k}(K_{3}: B_{3}^{+})$ is obtained. Thus the Gllai-Ramsey numbers for graphs with five vertices and eight edges are solved completely.
Furthermore, the the Gallai-Ramsey numbers $gr_{k}(K_{3}: r\cdot B_3^+,~
(k-r)\cdot S_3^+)$, $gr_{k}(K_{3}: r\cdot B_3^+,~
(k-r)\cdot K_3)$ and $gr_{k}(K_{3}: s\cdot S_3^+,~
(k-s)\cdot K_3)$ are obtained, respecticely.

\end{abstract}

\vspace{4mm}
\section{Introduction}

All graphs considered in this paper are finite, simple and undirected. For a graph $G$, we use $|G|$ to denote the number of vertices of $G$, say the \emph{order} of $G$.
The complete graph of order $n$ is denoted by $K_{n}$.
For a subset $S\subseteq V(G)$, let $G[S]$ be the subgraph of $G$ induced by $S$.
For two disjoint subsets $A$ and $B$ of $V(G)$, $E_{G}(A, B)=\{ab\in E(G) ~|~ a\in A, b\in B\}$. Let $G_1=(V_1,E_1)$ and $G_2=(V_2,E_2)$ be two graphs. The \emph{union} of $G_1$ and $G_2$, denoted by $G_{1}+G_{2}$, is the graph with the vertex set $V_1\bigcup V_2$ and the edge set $E_1\bigcup E_2$. The \emph{join} of $G_1$ and $G_2$, denoted by $G_1\vee G_2$, is the graph obtained from $G_1+G_2$ by adding all edges joining each vertex of $G_1$ and each vertex of $G_2$.
For any positive integer $k$, we write $[k]$ for the set $\{1, 2, \cdots, k\}$.
An edge coloring of a graph is called \emph{monochromatic} if all edges are colored by the same color. An edge-colored graph is called \emph{rainbow} if no two edges are colored by the same color.

Given graphs $H_{1}$ and $H_{2}$, the classical Ramsey number $R(H_{1}, H_{2})$ is the smallest integer $n$ such that for any 2-edge coloring of $K_{n}$ with red and blue, there exists a red copy of $H_{1}$ or a blue copy of $H_{2}$.
Given graphs $H_{1}, H_{2}, \cdots, H_{k}$, the multicolor Ramsey number $R(H_{1}, H_{2}, \cdots, H_{k})$ is the smallest positive integer $n$ such that for every $k$-edge colored $K_{n}$ with the color set $[k]$, there exists some $i\in [k]$ such that $K_{n}$ contains a monochromatic copy of $H_{i}$ colored by $i$.
The multicolor Ramsey number is an obvious generalization of the classical Ramsey number.
When $H=H_{1}=\cdots=H_{k}$, we simply denote $R(H_{1}, \cdots, H_{k})$ by $R_{k}(H)$. 
In this paper, we study Ramsey number in Gallai-coloring.
A \emph{Gallai-coloring} is an edge coloring of a complete graph without rainbow triangle.
Gallai-coloring naturally arises in several areas including: information theory \cite{JG}; the study of partially ordered sets, as in Gallai's original paper \cite{Gallai} (his result was restated in \cite{GyarfasSimonyi} in the terminology of graphs); and the study of perfect graphs \cite{KJLA}. 
A Gallai $k$-coloring is a Gallai-coloring that uses $k$ colors.
Given a positive integer $k$ and graphs $H_{1}, H_{2}, \cdots, H_{k}$, the \emph{Gallai-Ramsey number} $gr_{k}(K_{3}: H_{1}, H_{2}, \cdots, H_{k})$ is the smallest integer $n$ such that every Gallai $k$-colored $K_{n}$ contains a monochromatic copy of $H_{i}$ in color $i$ for some $i\in [k]$.
Clearly, $gr_{k}(K_{3}: H_{1}, H_{2}, \cdots, H_{k})\leq R(H_{1}, H_{2}, \cdots, H_{k})$ for any $k$ and $gr_{2}(K_{3}: H_{1}, H_{2})=R(H_{1}, H_{2})$.
When $H=H_{1}=\cdots=H_{k}$, we simply denote $gr_{k}(K_{3}: H_{1}, H_{2}, \cdots, H_{k})$ by $gr_{k}(K_{3}: H)$.
When $H=H_{1}=\cdots=H_{s} (0\leq s\leq k)$ and $G=H_{s+1}=\cdots=H_{k}$, we use the following shorthand notation $$gr_{k}(K_{3}:s\cdot H, (k-s)\cdot G)=gr_{k}(K_{3}:\underbrace{H, \cdots, H}_{s ~\text{times}}, \underbrace{G, \cdots, G}_{(k-s) ~\text{times}}).$$
For nonnegative integers $r, s, t$, when $G_1=H_{1}=\cdots=H_{r}$, $G_2=H_{r+1}=\cdots=H_{r+s}$,  and $G_3=H_{r+s+1}=\cdots=H_{r+s+t}$ with that $k=r+s+t$, we use the following shorthand notation $$gr_{k}(K_{3}:r\cdot G_1, s\cdot G_2, t\cdot G_3)=gr_{k}(K_{3}:\underbrace{G_1, \cdots, G_1}_{r ~\text{times}}, \underbrace{G_2, \cdots, G_2}_{s ~\text{times}}, \underbrace{G_3, \cdots, G_3}_{t ~\text{times}}).$$

The Gallai-Ramsey numbers $gr_k(K_3: H)$ for all the graphs $H$ on five vertices and at most seven edges are obtained (see \cite{2020Gallai, ZOU2019164, 2020Su, Zhao}).
There are two graphs on five vertices and eight edges, one of which is a wheel graph $W_4$, and the other is a graph $B_3^+$ obtained from the book graph $B_3$ by adding an edge between two vertices with degree two. Song \cite{SONG2020111725} and Mao \cite{2019Ramsey} obtained the Gallai-Ramsey number $gr_k(K_3: W_4)$.
In this paper, we determine the Gallai-Ramsey number $gr_k(K_3: B_3^+)$.
In order to get $gr_k(K_3: B_3^+)$, we actually investigate the Gallai-Ramsey number $gr_{k}(K_{3}:r\cdot B_{3}^{+}, s\cdot S_{3}^+, t\cdot K_3)$, where $S_3^+$ denotes the graph obtained from the star $S_3=K_{1}\vee\overline{K_3}$ by adding an edge between two pendant vertices, stated in Theorem~\ref{Thm:B3+K3}.

\begin{theorem}\label{Thm:B3+K3}
For nonnegative integers $r, s, t$, let $k=r+s+t$. Then
$$
gr_k(K_{3}:r\cdot B_{3}^{+}, s\cdot S_{3}^{+}, t\cdot K_3)=\begin{cases}
17^{\frac{r}{2}}\cdot 5^{\frac{t}{2}}+1, &\text{if $r, t$ are even, ($c_1$)}\\
2\cdot17^{\frac{r}{2}}\cdot 5^{\frac{t-1}{2}}+1, &\text{if $r$ is even, $t$ is odd, ($c_2$)}\\
8\cdot17^{\frac{r-1}{2}}\cdot 5^{\frac{t-1}{2}}+1, &\text{if $r, t$ are odd, ($c_3$)}\\
4\cdot17^{\frac{r-1}{2}}\cdot 5^{\frac{t}{2}}+1, &\text{if $r$ is odd, $t$ is even, ($c_{4}$)}\\
6\cdot17^{\frac{r}{2}}\cdot 5^{\frac{s+t-2}{2}}+1, &\text{if $r$ is even and $s+t$ is even, ($c_5$)}\\
3\cdot17^{\frac{r}{2}}\cdot 5^{\frac{s+t-1}{2}}+1, &\text{if $r$ is even and $s+t$ is odd, ($c_6$)}\\
48\cdot 17^{\frac{r-1}{2}}\cdot 5^{\frac{s+t-3}{2}}+1, &\text{if $r$ is odd and $t\geq1$ and $s+t$ is odd, ($c_7$)}\\
48\cdot 17^{\frac{r-1}{2}}\cdot 5^{\frac{s-3}{2}}+1, &\text{if $r, s$ are odd and $s\geq3$ and $t=0$, ($c_8$)}\\
9\cdot17^{\frac{r-1}{2}}+1, &\text{if $r$ is odd and $s=1$ and $t=0$, ($c_{9}$)}\\
24\cdot 17^{\frac{r-1}{2}}\cdot 5^{\frac{s+t-2}{2}}+1, &\text{if $r$ is odd and $s+t$ is even, ($c_{10}$)}\\
\end{cases}
$$
where $s=0$ for Condition $c_1$ to Condition $c_4$ and $s\geq1$ for Condition $c_5$ to Condition $c_{10}$.
\end{theorem}

When we set that $s=t=0$, $r=s=0$, $r=t=0$, respectively, we can get the following Theorem~\ref{Thm:B3+} to Theorem~\ref{Thm:K3+}, respectively. 

\begin{theorem}\label{Thm:B3+}
For any integer $r\geq1$, $$
gr_r(K_3: B_3^+)= \begin{cases}
17^{r/2}+1,  & \text{if $r$ is even,}\\
4 \cdot 17^{(r - 1)/2}+1,  & \text{if $r$ is odd.}
\end{cases}
$$
\end{theorem}

\begin{theorem}{\upshape \cite{Fre, AGAS}}\label{Thm:K3}
For a positive integer $t$,
$$
gr_t(K_{3}:K_{3})=\begin{cases}
5^{\frac{t}{2}}+1, &\text{if $t$ is even,}\\
2\cdot 5^{\frac{t-1}{2}}+1, &\text{if $t$ is odd.}
\end{cases}
$$
\end{theorem}

\begin{theorem}{\upshape \cite{Wang}}\label{Thm:K3+}
For a positive integer $s$,
$$
gr_s(K_{3}:S_{3}^+)=\begin{cases}
6\cdot5^{\frac{s-2}{2}}+1, &\text{if $s$ is even,}\\
3\cdot 5^{\frac{s-1}{2}}+1, &\text{if $s$ is odd.}
\end{cases}
$$
\end{theorem}

When we set $r=0$, $s=0$ and $t=0$, respectively, we can get the Gallai-Ramsey numbers $gr_{k}(K_{3}: s\cdot S_3^+,~
(k-s)\cdot K_3)$, $gr_{k}(K_{3}: r\cdot B_3^+,~
(k-r)\cdot K_3)$ and $gr_{k}(K_{3}: r\cdot B_3^+,~
(k-r)\cdot S_3^+)$, respectively, which are omitted.

To prove Theorem~\ref{Thm:B3+K3}, the following theorem is useful.

\begin{theorem}{\upshape \cite{Gallai, GyarfasSimonyi, CameronEdmonds}} {\upshape (Gallai-partition)}\label{Thm:G-Part}
For any Gallai-coloring of a complete graph $G$, there exists a partition of $V(G)$ into at least two parts such that there are at most two colors on the edges between the parts and there is only one color on the edges between each pair of parts. The partition is called a Gallai-partition.
\end{theorem}


\section{Proof of Theorem~\ref{Thm:B3+K3}}
First, recall some known classical Ramsey numbers which are useful.

\begin{lemma}{\upshape \cite{Clm, HHIM, SPR}}\label{Lem:B3+K3+}
\begin{align*}
R_{2}(K_{3})=6, R_2(S_3^+)=7, R_{2}(B_{3}^{+})=18, R(K_3, S_3^+)=7, R(K_3, B_3^+)=9, R(S_3^+, B_3^+)=10.
\end{align*}
\end{lemma}

A \emph{sharpness example} of the Ramsey number $R(H_{1}, H_{2})$, denoted by $C_{(H_{1}, H_{2})}$, is a 2-edge colored $K_{R(H_{1}, H_{2})-1}$ with red and blue such that there is neither red copy of $H_{1}$ nor blue copy of $H_{2}$.
For example, $C_{(K_3, K_3)}$ is a 2-edge colored $K_5$ with red and blue such that there is neither red nor blue copy of $K_3$ and $C_{(K_3, S_3^+)}$ is a 2-edge colored $K_6$ with red and blue such that there is neither red copy of $K_3$ nor blue copy of $S_3^+$.

For the sake of notation, let $f(r, s, t)= gr_k(K_{3}:r\cdot B_{3}^{+}, s\cdot S_{3}^{+}, t\cdot K_3)-1$,  claimed in Theorem~\ref{Thm:B3+K3}.
It is easy to check the following inequalities:
\begin{eqnarray}
& &\frac{f(r, s, t-1)}{f(r, s, t)}\leq \frac{1}{2},\label{equ:T1}
\end{eqnarray}
\begin{eqnarray}
& &\frac{f(r, s-1, t)}{f(r, s, t)}\leq \frac{1}{2},\label{equ:T2}
\end{eqnarray}
\begin{eqnarray}
& &\frac{f(r-1, s, t)}{f(r, s, t)}\leq \begin{cases}
\frac{1}{3}, &\text{if $s=1$ and $t=0$,}\\
\frac{5}{16}, &\text{others,}\\
\end{cases}\label{equ:T7}
\end{eqnarray}
\begin{eqnarray}
& &\frac{f(r-1, s+1, t)}{f(r, s, t)}\leq \frac{3}{4},\label{equ:T3}
\end{eqnarray}
\begin{eqnarray}
& &\frac{f(r-1, s, t+1)}{f(r, s, t)}\leq \frac{2}{3},\label{equ:T6}
\end{eqnarray}
\begin{eqnarray}
& &\frac{f(r, s, t-2)}{f(r, s, t)}\leq \frac{1}{5},\label{equ:T8}
\end{eqnarray}
\begin{eqnarray}
& &\frac{f(r, s-1, t-1)}{f(r, s, t)}\leq \frac{1}{5},\label{equ:T9}
\end{eqnarray}
\begin{eqnarray}
& &\frac{f(r, s-2, t)}{f(r, s, t)}\leq \frac{1}{5},\label{equ:T10}
\end{eqnarray}
\begin{eqnarray}
& &\frac{f(r-1, s, t-1)}{f(r, s, t)}\leq \frac{1}{8},\label{equ:T4}
\end{eqnarray}
\begin{eqnarray}
& &\frac{f(r-1, s-1, t)}{f(r, s, t)}\leq \frac{1}{8}, \label{equ:T12}
\end{eqnarray}
\begin{eqnarray}
& &\frac{f(r-1, s+1, t-1)}{f(r, s, t)}\leq \frac{3}{8},\label{equ:T5}
\end{eqnarray}
\begin{eqnarray}
& &\frac{f(r-1, s-1, t+1)}{f(r, s, t)}\leq \frac{5}{16}, \label{equ:T13}
\end{eqnarray}
\begin{eqnarray}
& &\frac{f(r-2, s+1, t+1)}{f(r, s, t)}\leq \frac{15}{34}, \label{equ:T14}
\end{eqnarray}
\begin{eqnarray}
& &\frac{f(r-2, s+1, t)}{f(r, s, t)}\leq \frac{3}{17}, \label{equ:T16}
\end{eqnarray}
\begin{eqnarray}
& &\frac{f(r-2, s, t+2)}{f(r, s, t)}\leq \frac{16}{51},\label{equ:T15}
\end{eqnarray}
\begin{eqnarray}
& &\frac{f(r-2, s, t+1)}{f(r, s, t)}\leq \frac{8}{51},\label{equ:T11}
\end{eqnarray}
\begin{eqnarray}
& &\frac{f(r-2, s, t)}{f(r, s, t)}=\frac{1}{17} .\label{equ:T17}
\end{eqnarray}

Now we prove Theorem~\ref{Thm:B3+K3}.

\begin{proof}
We first prove that $gr_k(K_{3}:r\cdot B_{3}^{+}, s\cdot S_{3}^+, t\cdot K_3)\geq f(r, s, t)+1$ by constructing a Gallai $k$-colored complete graph $G_k$ with order $f(r, s, t)$ which contains no monochromatic copy of $B_{3}^{+}$ colored by one of the first $r$ colors and no monochromatic copy of $S_{3}^+$ colored by one of the middle $s$ colors and no monochromatic copy of $K_{3}$ colored by one of the remaining $t$ colors.
For this construction, we use the sharpness example of classical Ramsey results.
Let $Q_{1}=C_{(K_{3}, K_{3})}$, $Q_{2}=C_{(K_{3}, S_{3}^{+})}$, $Q_{3}=C_{(K_{3}, B_{3}^{+})}$, $Q_{4}=C_{(S_{3}^+, S_{3}^{+})}$, $Q_{5}=C_{(S_{3}^+, B_{3}^{+})}$ and $Q_{6}=C_{(B_{3}^+, B_{3}^{+})}$.
We construct our sharpness example $G_k$ by taking blow-ups of these sharpness examples $Q_{j}~ (j\in [6])$.
A \emph{blow-up} of a 2-edge-colored graph $Q_j$ with two new colors on an $i$-edge-colored graph $G_i$ is a new graph $G_{i+2}$ obtained from $Q_j$ by replacing each vertex of $Q_j$ with $G_i$ and replacing each edge $e$ of $Q_j$ with a monochromatic complete bipartite graph $(V(G_i), V(G_i))$ in the same color with $e$.
By induction on $i$, 
suppose that we have constructed graphs $G_i$, where $G_i$ is $i$-edge-colored such that $G_i$ contains no rainbow triangle and no appropriately colored monochromatic $B_3^+$ or $S_3^+$ or $K_3$.
If $i=k$, then the construction is completed.
Otherwise, we construct $G_{i+2}$ by taking a blow-up of 2-edge-colored $Q_j$ with two new colors on $G_i$ by distinguishing the following cases. \\
{\bf Case a.} If the two new colors are in the first $r$ colors, then we construct $G_{i+2}$ by making a blow-up of $Q_6$ on $G_i$. \\
{\bf Case b.} If the two new colors are in the middle $s$ colors, then we construct $G_{i+2}$ by making a blow-up of $Q_1$ on $G_i$ where $i\geq1$. \\
{\bf Case c.} If the two new colors are in the last $t$ colors, then we construct $G_{i+2}$ by making a blow-up of $Q_1$ on $G_i$.

The base graphs (i.e., the first graphs in the induction) for this construction are constructed as follows.\\
For {\bf Condition $c_1$}, the base graph $G_0$ is a single vertex.\\
For {\bf Condition $c_2$}, the base graph $G_1$ is a $K_2$ colored by one of the last $t$ colors.\\
For {\bf Condition $c_3$}, the base graph $G_2$ is a $Q_3$ colored by two colors in which one is in the first $r$ colors and the other is in the last $t$ colors.\\
For {\bf Condition $c_4$}, the base graph $G_1$ is a monochromatic $K_4$ colored by one of the first $r$ colors.\\
For {\bf Condition $c_5$}, the base graph $G_2$ is a $Q_4$ colored by two colors which are in the middle $s$ colors if $s, t$ are both even and the base graph $G_2$ is a $Q_2$ colored by two colors in which one is in the middle $s$ colors and the other is in the last $t$ colors if $s, t$ are both odd.\\
For {\bf Condition $c_6$}, if $s$ is odd and $t$ is even, the base graph $G_1$ is a monochromatic copy of $K_3$ colored by one of the middle $s$ colors.
If $s$ is even and $t$ is odd, the base graph $G_3$ is a blow-up of $Q_1$ on a monochromatic $K_3$, where $K_3$ is colored by one of the middle $s$ colors and $Q_1$ is colored by two new colors one of which is in the middle $s$ colors and the other is in the last $t$ colors.\\
For {\bf Condition $c_7$}, if $s$ is even and $t$ is odd, the base graph $G_4$ is a blow-up of $Q_3$ on $Q_4$, where $Q_4$ is colored by two colors which are in the middle $s$ colors and $Q_3$ is colored by two colors one of which is in the first $r$ colors and the other is in the last $t$ colors .
If $s$ is odd and $t$ is even, then we first construct $G_3$. $G_3$ is a blow-up of $Q_3$ on a monochromatic $K_3$, where $K_3$ is colored by one of middle $s$ colors and $Q_3$ is colored by two colors one of which is in the first $r$ colors and the other is in the last $t$ colors. The base graph $G_4$ is a blow-up of a monochromatic $K_2$ on $G_3$, where $K_2$ is colored by a new color in the last $t$ colors.\\
For {\bf Condition $c_8$}, the base graph $G_4$ is a blow-up of $Q_3$ on $Q_4$, where $Q_4$ is colored by two colors which are in the middle $s$ colors and $Q_3$ is colored by two new colors one of which is in the first $r$ colors and the other is in the middle $s$ colors.\\
For {\bf Condition $c_9$}, the base graph $G_2$ is a $Q_5$ with two colors in which one is in the first $r$ colors and the other is in the middle $s$ colors.\\
For {\bf Condition $c_{10}$}, if $s$ and $t$ are both odd, the base graph $G_3$ is a blow-up of $Q_3$ on a monochromatic $K_3$, where $K_3$ is colored by one of the middle $s$ colors and $Q_3$ is colored by two colors one of which is in the first $r$ colors and the other is in the last $t$ colors.
If $s$ and $t$ are even, the base graph $G_3$ is a blow-up of $Q_3$ on a monochromatic $K_3$, where $K_3$ is colored by one of the middle $s$ colors and $Q_3$ is colored by two new colors one of which is in the first $r$ colors and the other is in the middle $s$ colors.\\
Now we only check $G_k$ constructed by the method as above under the two conditions $c_1$ and $c_5$, others can be checked similarly, omitted. \\
{\bf Condition $c_1$.} $r, t$ are even and $s=0$. Then $k=r+t$. First the basic graph is $G_0=K_1$. Next by Case a, we can construct $G_r$ of order $17^{\frac{r}{2}}$. Finally, by Case c, we continue to construct $G_{r+2}, G_{r+4}, \ldots, G_{r+t}$. So we can get that $|G_k|=17^{\frac{r}{2}}\cdot5^{\frac{t}{2}}=f(r, s, t)$ and $G_k$ is a Gallai $k$-colored complete graph which contains neither monochromatic copy of $B_{3}^{+}$ colored by one of the first $r$ colors nor monochromatic copy of $K_{3}$ colored by one of the remaining $t$ colors.\\
{\bf Condition $c_5$.} If $r$ is even and $s, t$ are old, then we first have the basic graph $G_2=K_6$. Next by Case a, we can construct $G_{r+2}$ of order $6\cdot17^{\frac{r}{2}}$. Then by Case b, we continue to construct $G_{r+s+1}$ of order $6\cdot17^{\frac{r}{2}}\cdot5^{\frac{s-1}{2}}$. Finally, by Case c, we continue to construct $G_{r+s+3}, G_{r+s+5}, \ldots, G_{r+s+t}$. So we can get that $|G_k|=6\cdot17^{\frac{r}{2}}\cdot5^{\frac{s+t-2}{2}}=f(r, s, t)$ and $G_k$ is a Gallai $k$-colored complete graph which contains no appropriately colored monochromatic $B_3^+$ or $S_3^+$ or $K_3$.
Similarly, we can get that $|G_{k}|=f(r, s, t)$ if $r$ is even and $s, t$ are even.\\
Therefore, $gr_k(K_{3}:r\cdot B_{3}^{+}, s\cdot
S_{3}^+, t\cdot K_{3})\geq f(r, s, t)+1$.


Now we prove that $gr_k(K_{3}:r\cdot B_{3}^{+}, s\cdot
S_{3}^+, t\cdot K_{3})\leq f(r, s, t)+1$ by induction on $3r+2s+t$.
The statement holds in the case that $3r+2s+t\leq 2$ or $k\leq2$ by Lemma~\ref{Lem:B3+K3+}.
So we can assume that $k\geq3, 3r+2s+t\geq3$, and the statement holds for any $r^{'}$, $s^{'}$ and $t^{'}$ such that $3r^{'}+2s^{'}+t^{'}<3r+2s+t$.
Let $n=f(r, s, t)+1$ and $G$ be a Gallai $k$-colored complete graph of order $n$.
Then $n=f(r, s, t)+1\geq f(0, 0, 3)+1=11$.
Suppose, to the contrary, that $G$ contains neither monochromatic copy of $B_{3}^{+}$ in any one of the first $r$ colors nor monochromatic copy of $S_{3}^+$ in any one of the middle $s$ colors nor monochromatic copy of $K_3$ in any one of the last $t$ colors.
By Theorem~\ref{Thm:G-Part}, there exists a Gallai-partition of $V(G)$. Choose a Gallai-partition with the smallest number of parts, say $(V_{1}, V_{2}, \cdots, V_{q})$ and let each part $H_{i}=G[V_{i}]$ for $i\in [q]$.
Then $q\ge 2$.

We first consider the case that $q=2$. W.L.O.G, suppose that the color on the edges between two parts is red.
First suppose that red is in the last $t$ colors or in the middle $s$ colors.
Then $H_{1}$ and $H_{2}$ both have no red edges, otherwise, there exists a red $K_{3}$ or a red $S_3^+$, a contradiction.
Hence by the induction hypothesis, $|H_{i}|\leq f(r, s, t-1)$ if red is in the last $t$ colors and $|H_{i}|\leq f(r, s-1, t)$ red is in the middle $s$ colors for each $i\in [2]$.
By Inequalities (\ref{equ:T1}) and (\ref{equ:T2}), we have that $$|G|=|H_{1}|+|H_{2}|\leq 2\times \frac{1}{2}f(r, s, t)<n,$$ a contradiction.
Next suppose that red is in the first $r$ colors.
If both $H_{1}$ and $H_{2}$ have a red edge, then $G$ has a red $B_3^+$, a contradiction.
First suppose that $H_1$ and $H_2$ both have no red edges.
By the induction hypothesis and Inequality (\ref{equ:T7}),
$$|G|=|H_{1}|+|H_{2}|\leq 2f(r-1, s, t)\leq\frac{2}{3} f(r, s, t)<n,$$ a contradiction.
Then suppose that $H_1$ has a red edge, but $H_2$ has no red edges.
Clearly, $H_1$ contains no red $S_3^+$, otherwise, $G$ contains a red $B_3^+$, a contradiction.
We first consider the case that $H_1$ contains a red $K_3$. To avoid a red $B_3^+$, we have that $|H_2|=1$.
Then $H_1$ can be considered that red is moved in the middle $s$ colors.
By the induction hypothesis and Inequality (\ref{equ:T3}), we get that
$$|G|=|H_{1}|+|H_{2}|\leq f(r-1, s+1, t)+1\leq\frac{3}{4} f(r, s, t)+1<n,$$ a contradiction.
So we can assume that $H_1$ contains no red $K_3$.
Thus $H_1$ can be considered that red is moved in the last $t$ colors.
By the induction hypothesis and Inequalities (\ref{equ:T6}) and (\ref{equ:T7}), we get that
$$|G|=|H_{1}|+|H_{2}|\leq f(r-1, s, t+1)+f(r-1, s, t)\leq \left(\frac{2}{3}+\frac{1}{3}\right) f(r, s, t)<n,$$ a contradiction.

Now we can assume that $q\geq 3$ and the two colors appeared in the Gallai-partition $(V_{1}$, $V_{2}$, $\cdots$, $V_{q})$ are red and blue.
If there exists one part (say $V_{1}$) such that all edges joining $V_{1}$ to the other parts are colored by the same color, then we can find a new Gallai-partition with two parts $(V_{1}, V_{2}\bigcup\cdots \bigcup V_{q})$, which contradicts with that $q$ is smallest.
It follows that $q\neq 3$ and the following fact holds.

\begin{fact}\label{fact:dwi}
For each part $V_{i}$, there exist both red and blue edges connecting $V_i$ and the other parts.
\end{fact}

Now we can assume that $q\ge 4$.
First suppose that red or blue are in the middle $s$ colors. Then we have the following
facts.

\begin{fact}\label{fact:red k3} If red is in the middle $s$ colors, then $G$ contains no red $K_{3}$ and the statement holds for blue symmetrically.
\end{fact}

Otherwise, suppose that there is a red $K_{3}=v_1 v_2 v_3$ in $G$.
Let $U$ be the union of parts $V_j$ containing a vertex in $\{v_1, v_2, v_3\}$.
Then $U$ contains at most $3$ parts of Gallai-partition in $G$.
If there exists a red edge in $E_{G}(U, V(G)\setminus U)$, then $G$ contains a red $S_{3}^{+}$, a contradiction.
It follows that all edges in $E_{G}(U, V(G)\setminus U)$ are blue.
Then $(U, V(G)\setminus U)$ is a new Gallai-partition which contradicts with that $q\geq4$ and $q$ is smallest.

By Fact~\ref{fact:dwi} and Fact~\ref{fact:red k3}, we have the following Fact.
\begin{fact}\label{new}
If red is in the middle $s$ colors, then every $H_{i}$ has no red edges and the statement holds
for blue symmetrically.
\end{fact}

Now we consider the following cases.

\begin{case}\label{case:r=0}
Neither red nor blue is in the first $r$ colors.
\end{case}

First we prove the following claim.
\begin{claim}\label{Cla:st}
$G$ contains neither red $K_3$ nor blue $K_3$.
\end{claim}

\begin{proof}
If red and blue are both in the last $t$ colors, the statement holds clearly.
Next suppose that both red and blue are within the middle $s$ colors. By Fact~\ref{fact:red k3}, the statement holds. Finally, suppose that red appears in the middle $s$ colors, blue appears in the last $t$ colors, W.L.O.G..
Then $G$ has no blue $K_3$.
By Fact~\ref{fact:red k3}, we have that $G$ contains no red $K_{3}$.
Complete the proof of Claim~\ref{Cla:st}.
\end{proof}

Since $R_2(K_{3})=6$, we have $q\leq5$ by Claim~\ref{Cla:st}.
By Fact~\ref{fact:dwi} and Claim~\ref{Cla:st}, each $H_{i}$ contains neither red nor blue edges.
By the induction hypothesis, for each $H_i$, we have that
$$
|H_i|\leq\begin{cases}
f(r, s, t-2), \text{if red and blue are both in the last $t$ colors,}\\
f(r, s-2, t), \text{if red and blue are both in the middle $s$ colors,}\\
f(r, s-1, t-1), \text{if red is in the middle $s$ colors and blue in the last $t$ colors.}
\end{cases}
$$
By Inequalities (\ref{equ:T8})-(\ref{equ:T10}), we have that
\begin{eqnarray*}
|G|=\sum_{i=1}^{q}|H_{i}| \leq 5\times \frac{1}{5}f(r, s, t) < n,
\end{eqnarray*}
a contradiction.
The proof of Case~\ref{case:r=0} is completed.

Now we consider the case that red or blue are in the first $r$ colors, we have the following fact and claim.

\begin{fact}\label{fact:red edges}
If red is in the first $r$ colors and both $H_{i}$ and $H_{j}$ contain red edges, then the edges in $E_{G}(V_{i}, V_{j})$ must be blue and the statement holds for blue symmetrically.
\end{fact}

Otherwise, suppose that the edges in $E_{G}(V_{i}, V_{j})$ are red.
To avoid a red $B_3^+$, all edges in $E_G(\{V_{i}, V_{j}\}, V(G)\setminus\{V_{i}, V_{j}\})$ are blue.
Then $(V_i\cup V_j, V(G)\setminus(V_i\cup V_j))$ is a new Gallai-partition which contradicts with that $q\geq4$ and $q$ is smallest.




\begin{claim}\label{Cla:B3+K3++}
Let red be in the first $r$ colors and $X$ is the union of $p$ parts of the Gallai-partition such that $G[X]$ contains neither red $B_3^+$ nor blue edges. Then \\
(i) $|X|\leq f(r-1, s, t)+f(r-1, s, t-1)$ if blue is in the last $t$ colors.\\
(ii) $|X|\leq f(r-1, s-1, t+1)+f(r-1, s-1, t)$ if blue is in the middle $s$ colors.\\
(iii) $|X|\leq f(r-2, s, t+1)+f(r-2, s, t)$ if blue is in the first $r$ colors.
\end{claim}

\begin{proof}
Since $G[X]$ contains no blue edges, every pair of parts in $G[X]$ are joined by red edges.
Since $G[X]$ has no red $B_3^+$, we have that $p\leq 4$.
By Fact~\ref{fact:red edges}, there is at most one part in $G[X]$ containing red edges.
Then we distinguish two cases.\\
{\bf Case I.} Each part in $G[X]$ has no red edges.\\
If $p\leq 3$, then $$
|X|\leq\begin{cases}
3f(r-1, s, t-1), \text{if blue is in the last $t$ colors,}\\
3f(r-1, s-1, t), \text{if blue is in the middle $s$ colors,}\\
3f(r-2, s, t), \text{if blue is in the first $r$ colors.}
\end{cases}
$$
By the definition of $f(r, s, t)$, we can check that $3f(r-1, s, t-1)\leq f(r-1, s, t)+f(r-1, s, t-1)$, $3f(r-1, s-1, t)\leq f(r-1, s-1, t+1)+f(r-1, s-1, t)$ and $3f(r-2, s, t)\leq f(r-2, s, t+1)+f(r-2, s, t)$. The statement of Claim~\ref{Cla:B3+K3++} holds in this case.
If $p=4$, to avoid a red $B_3^+$, then we have that $|X|=4$.
Clearly, the statement still holds.\\
{\bf Case II.} There is a unique part contains a red edge in $G[X]$, say $H_1$.\\
If $H_1$ contains a red $K_3$, then to avoid a red $B_3^+$, $|X\backslash V_1|\leq 1$.
So $$
|X|\leq |H_1|+1\leq\begin{cases}
f(r-1, s+1, t-1)+1, \text{if blue is in the last $t$ colors,}\\
f(r-1, s, t)+1, \text{if blue is in the middle $s$ colors,}\\
f(r-2, s+1, t)+1, \text{if blue is in the first $r$ colors.}
\end{cases}
$$
In this case, the statement of Claim~\ref{Cla:B3+K3++} holds.
If $H_1$ contains no red $K_3$, then $$
|H_1|\leq\begin{cases}
f(r-1, s, t), \text{if blue is in the last $t$ colors,}\\
f(r-1, s-1, t+1), \text{if blue is in the middle $s$ colors,}\\
f(r-2, s, t+1), \text{if blue is in the first $r$ colors.}
\end{cases}
$$
Since $G[X]$ has no red $B_3^+$, we have that $p\leq 3$.
Furthermore, $|X|=|H_1|+2$ if $p=3$ and $|X|=|H_1|$ if $p=1$. It is easy to check that the statement also holds.
If $p=2$, then the other part in $G[X]$, say $H_2$, contains no red edges.
So $$
|X|=|H_1|+|H_2|\leq\begin{cases}
f(r-1, s, t)+f(r-1, s, t-1), \text{if blue is in the last $t$ colors,}\\
f(r-1, s-1, t+1)+f(r-1, s-1, t), \text{if blue is in the middle $s$ colors,}\\
f(r-2, s, t+1)+f(r-2, s, t), \text{if blue is in the first $r$ colors.}
\end{cases}
$$
Complete the proof of Claim~\ref{Cla:B3+K3++}.
\end{proof}

For the part $V_1$ of Gallai-partition, let $V_R$ $(V_B)$ be the union of parts $V_i$ such that the edges in $E_G(V_1, V_i)$ are red (blue) and $H_R=G[V_R]$, $H_B=G[V_B]$.

\begin{case}\label{case:r>=1}
Exactly one of red and blue is in the first $r$ colors. 
\end{case}

W.L.O.G., suppose that red appears in the first $r$ colors. Then $G$ has no red $B_3^+$.
If blue appears in the last $t$ colors,
then there is no blue $K_{3}$ in $G$.
If blue appears in the middle $s$ colors, then by Fact~\ref{fact:red k3}, we get that there is also no blue $K_3$ in $G$.
Since $R(B_3^+, K_3)=9$, $q\leq 8$.
By Fact~\ref{fact:dwi}, each $H_i$ contains neither red $S_3^+$ nor blue edges.
If each $H_i$ contains no red edges, then by the induction hypothesis and Inequalities (\ref{equ:T4}) and (\ref{equ:T12}), we get that
$$
|G|=\sum_{i=1}^{q}|H_{i}|\leq\begin{cases}
8 f(r-1, s, t-1) \leq f(r, s, t) < n, \text{if blue is in the last $t$ colors,}\\
8 f(r-1, s-1, t) \leq f(r, s, t) < n, \text{if blue is in the middle $s$ colors.}
\end{cases}
$$
a contradiction.
It follows that there is at least one part containing red edges.
W.L.O.G., suppose that $H_1$ contains a red edge.
So we consider the following two Subcases.

\begin{subcase}\label{case:2.1}
$H_1$ has a red $K_3$.
\end{subcase}
To avoid a red $B_3^+$ or a blue $K_3$, we get that $|H_R|=1$ and $H_B$ contains no blue edges.
By the induction hypothesis and Inequalities (\ref{equ:T1}) and (\ref{equ:T5}), (\ref{equ:T2}) and (\ref{equ:T7}), we get that
\begin{eqnarray*}
|G|
&=&|H_{1}|+|H_{R}|+|H_{B}| \\
&\leq & \begin{cases}
f(r-1, s+1, t-1)+1+f(r, s, t-1)\leq\frac{7}{8} f(r, s, t)+1<n, \text{if blue is in the last $t$ colors,}\\
f(r-1, s, t)+1+f(r, s-1, t)\leq\frac{5}{6} f(r, s, t)+1<n, \text{if blue is in the middle $s$ colors.}
\end{cases}
\end{eqnarray*}
a contradiction.

\begin{subcase}\label{case:2.2}
$H_1$ has no red $K_3$.
\end{subcase}

To avoid a blue $K_3$, $H_B$ contains no blue edges. By Claim~\ref{Cla:B3+K3++}, we have that
$$
|H_B|\leq\begin{cases}
f(r-1, s, t)+f(r-1, s, t-1), \text{if blue is in the last $t$ colors,}\\
f(r-1, s-1, t+1)+f(r-1, s-1, t), \text{if blue is in the middle $s$ colors.}
\end{cases}
$$
To avoid a red $B_3^+$, $H_R$ contains no red edges if $|H_R|\geq 3$.
It follows that the edges between each pair of parts in $H_R$ are blue.
Since $G$ has no blue $K_3$, $H_R$ contains at most two parts of Gallai-partition and each part in $H_R$ contains no blue edges. By the induction hypothesis,
\begin{eqnarray}
|H_{R}|
\leq\begin{cases}
2f(r-1, s, t-1), \text{if blue is in the last $t$ colors,}\\
2f(r-1, s-1, t), \text{if blue is in the middle $s$ colors.}
\end{cases}\label{equ:T18}
\end{eqnarray}
If $|H_R|\leq 2$, then the Inequality (\ref{equ:T18}) still holds.
By the induction hypothesis and Inequalities (\ref{equ:T4}) and (\ref{equ:T7}), (\ref{equ:T12}) and (\ref{equ:T13}), we get that
\begin{eqnarray*}
|G|
&=&|H_{1}|+|H_{B}|+|H_{R}| \\
&\leq & \begin{cases}
2f(r-1, s, t)+3f(r-1, s, t-1)\leq f(r, s, t)<n, \text{if blue is in the last $t$ colors,}\\
2f(r-1, s-1, t+1)+3f(r-1, s-1, t)<n, \text{if blue is in the middle $s$ colors.}
\end{cases}
\end{eqnarray*}
a contradiction.
The proof of Case~\ref{case:r>=1} is completed.

\begin{case}\label{case:r>=2}
Both red and blue are in the first $r$ colors.
\end{case}

In this case, the graph $G$ contains neither red nor blue $B_{3}^+$.
Since $R_2(B_{3}^+)=18$, we have $q\leq17$.
First we prove the following claim.

\begin{claim}\label{Cla:rK3bK3}
Each part $H_i$ of $G$ contains neither red nor blue $K_3$.
\end{claim}

\begin{proof}
To the contrary, suppose that there exist one part, say $H_1$, containing a red $K_3$ and a blue $K_3$.
To avoid a red or blue $B_3^+$, we have that $|H_R|=|H_B|=1$.
So we get $q=3$, which contradicts with that $q\geq4$.
It follows that each part of $G$ at most contains a red or a blue $K_3$.
W.L.O.G., suppose that there is a part $H_1$ containing a blue $K_3$ but no red $K_3$.
To avoid a blue $B_3^+$, we get that $|H_B|=1$.
If $H_1$ contains no red edges, then to avoid a red $B_3^+$, $H_R$ contains no red $K_3$.
By induction hypothesis and Inequalities (\ref{equ:T6}) and (\ref{equ:T16}), we have that
\begin{eqnarray*}
|G|
=|H_1|+|H_R|+|H_B| \leq f(r-2, s+1, t)+f(r-1, s, t+1)+1
\leq  \frac{43}{51}f(r, s, t)+1<n,
\end{eqnarray*}
a contradiction.
Now we assume that $H_1$ contains a red edge. To avoid a red $B_3^+$, $H_R$ contains no red edges if $|H_R|\geq 3$. By induction hypothesis, $|H_R|\leq f(r-1, s, t)$. If $|H_R|\leq 2$, then $|H_R|\leq f(r-1, s, t)$ clearly.
By Inequalities (\ref{equ:T7}) and (\ref{equ:T14}), we have that
\begin{eqnarray*}
|G|
=|H_1|+|H_R|+|H_B|
\leq  f(r-2, s+1, t+1)+ f(r-1, s, t)+1\leq \frac{79}{102}f(r, s, t)+1<n,
\end{eqnarray*}
a contradiction.
Complete the proof of Claim~\ref{Cla:rK3bK3}.
\end{proof}
Now we consider the following subcases.


\begin{subcase}\label{case:3.1}
There exists a part $H_1$ containing both a red and a blue edge.
\end{subcase}

To avoid a blue or red $B_3^+$, $H_R$ contains no red edges and $H_B$ contains no blue edges if $|H_R|, |H_B|\geq 3$. By induction hypothesis, we have that $|H_R|, |H_B|\leq f(r-1, s, t)$.
If $|H_R|, |H_B|\leq 2$, then $|H_R|, |H_B|\leq f(r-1, s, t)$ clearly.
By Inequalities (\ref{equ:T7}) and (\ref{equ:T15}), we have that
\begin{eqnarray*}
|G|
=|H_1|+|H_R|+|H_B|
\leq  f(r-2, s, t+2)+ 2f(r-1, s, t)\leq \frac{50}{51}f(r, s, t)<n,
\end{eqnarray*}
a contradiction.

Now we can assume that each part contains no red edges or no blue edges in the following.
We call a part \emph{free} if it contains neither red nor blue edges.
We call a part \emph{red (blue)} if it contains only red (blue) edges.
By induction hypothesis and Claim~\ref{Cla:rK3bK3}, we have that $|H_i|\leq f(r-2, s, t)$ if $H_i$ is a free part and $|H_i|\leq f(r-2, s, t+1)$ if $H_i$ is a red or blue part.

\begin{subcase}\label{case:3.2}
Each part $H_i$ is a free part.
\end{subcase}
By Inequality (\ref{equ:T17}), we have that
\begin{eqnarray*}
|G|=
\sum_{i=1}^{q}|H_{i}| \leq 17 f(r-2, s, t) =f(r, s, t)<n,
\end{eqnarray*}
a contradiction.
\begin{subcase}\label{case:3.2}
There exists one part $H_1$ which is a red or blue part.
\end{subcase}

W.L.O.G., suppose that $H_1$ is a red part.
If $|H_R|\geq 3$, then $H_R$ has no red edges since $G$ has no red $B_3^+$.
Then by Claim ~\ref{Cla:B3+K3++}, we have that $|H_R|\leq f(r-2, s, t)+f(r-2, s, t+1)$.
If $|H_R|\leq 2$, then $|H_R|\leq 2\leq f(r-2, s, t)+f(r-2, s, t+1)$ clearly.
To avoid a blue $B_3^+$, $H_B$ contains no blue $K_3$.
Since $H_B$ contains neither blue $K_3$ nor red $B_3^+$ and $R(K_3, B_3^+)=9$, we have that $H_B$ contains at most 8 parts of Gallai-partition. First suppose that all parts in $H_B$ are free parts. Then $|H_B|\leq 8 f(r-2, s, t)$.
By Inequalities (\ref{equ:T11}) and (\ref{equ:T17}), we have that
\begin{eqnarray*}
|G|
=|H_1|+|H_R|+|H_B| \leq 2f(r-2, s, t+1)+9f(r-2, s, t)
\leq \frac{35}{51}f(r, s, t)<n,
\end{eqnarray*}
a contradiction.
Next suppose that there is one part, say $H_2$, such that $H_2 (H_2\subseteq H_B)$ is a blue or red part.
let $V_{2R}$ $(V_{2B})$ be the union of parts $H_i$ such that the edges in $E_G(V_2, V_i)$ are red (blue) and $H_{2R}=G[V_{2R}]$, $H_{2B}=G[V_{2B}]$.
If $H_2$ is a blue part, then to avoid a red or blue $B_3^+$, $H_{2R}$ contains neither red nor blue $K_3$ and $H_{2B}=\emptyset$.
By the induction hypothesis, $|H_B|=|H_2|+|H_{2R}|\leq f(r-2, s, t+1)+ f(r-2, s, t+2)$.
By Inequalities (\ref{equ:T11}), (\ref{equ:T15}) and (\ref{equ:T17}), we have that
\begin{eqnarray*}
|G|
&=&|H_1|+|H_R|+|H_B| \leq 3f(r-2, s, t+1)+f(r-2, s, t)+f(r-2, s, t+2)\\
&\leq& \frac{43}{51}f(r, s, t)<n,
\end{eqnarray*}
a contradiction.
If $H_2$ is a red part, then $H_{2B}$ contains no blue edges since $G$ contains no blue $B_3^+$.
Then by Claim~\ref{Cla:B3+K3++}, $|H_{2B}|\leq f(r-2, s, t+1)+f(r-2, s, t)$.
To avoid a red $B_3^+$, $H_{2R}$ contains no red edge if $|H_{2R}|\geq 3$.
It follows that the edges between each pair of parts in $H_{2R}$ are blue.
Since $G$ has no blue $B_3^+$, $H_{2R}$ contains no blue $K_3$.
It follows that there are at most two parts in $H_{2R}$.
If $H_{2R}$ contains only one part, then $H_{2R}$ is a free part or a blue part.
We can get that $|H_{2R}|\leq f(r-2, s, t+1)$.
If $H_{2R}$ contains two parts, then two parts both are free parts.
We have that $|H_{2R}|\leq 2f(r-2, s, t)\leq f(r-2, s, t+1)$.
So \begin{eqnarray*}
|H_B|
=|H_2|+|H_{2R}|+|H_{2B}| \leq 3f(r-2, s, t+1)+f(r-2, s, t).
\end{eqnarray*}
By Inequalities (\ref{equ:T11}) and (\ref{equ:T17}), we have that
\begin{eqnarray*}
|G|
=|H_1|+|H_R|+|H_B| \leq 5f(r-2, s, t+1)+2f(r-2, s, t)
\leq \frac{46}{51}f(r, s, t)<n,
\end{eqnarray*}
a contradiction.
Complete the proof of Case~\ref{case:r>=2} and then the proof of Theorem~\ref{Thm:B3+K3}.
\end{proof}

\bibliography{bibfile1}

\end{document}